\documentclass[12pt]{amsart}
\usepackage{fullpage}
\usepackage{ulem}
\usepackage{amsmath}
\usepackage{amsthm}
\usepackage{amsbsy}
\usepackage{amsrefs}
\usepackage{amssymb}
\usepackage{hyperref}
\usepackage{amsfonts}

\usepackage{enumitem}
\usepackage{amsfonts}

\usepackage{pbox,float,color}
\usepackage{graphicx}
\usepackage{float}
\newtheorem{theorem}{Theorem}
\newtheorem{lemma}{Lemma}
\newtheorem{corollary}{Corollary}

\theoremstyle{remark}
\newtheorem*{remark}{Remark}

\newcommand{\lam}{\lambda}
\renewcommand{\phi}{\varphi}

\newcommand{\Gam}{\Gamma}

\numberwithin{equation}{section}

\begin{document}

\title[Benford's Law for Coefficients of Newforms]
{Benford's Law for Coefficients of Newforms}

\author{Marie Jameson} 

\author{Jesse Thorner} 

\author{Lynnelle Ye} 

\address{Marie Jameson, Department of Mathematics, University of Tennessee, Knoxville, TN 37996}
\email{marie.jameson@gmail.com}
\address{Jesse Thorner, Department of Mathematics and Computer Science, Emory University, Atlanta, Georgia 30322}
\email{jesse.thorner@gmail.com}
\address{Lynnelle Ye, Department of Mathematics, Harvard University, Cambridge, MA 02138}
\email{lye@g.harvard.edu}

\begin{abstract}
Let $f(z)=\sum_{n=1}^\infty\lambda_f(n)e^{2\pi i n z}\in S_{k}^{\textup{new}}(\Gamma_0(N))$ be a newform of even weight $k\geq2$ on $\Gamma_0(N)$ without complex multiplication.  Let $\mathbb{P}$ denote the set of all primes.  We prove that the sequence $\{\lambda_f(p)\}_{p\in\mathbb{P}}$ does not satisfy Benford's Law in any integer base $b\geq2$.  However, given a base $b\geq2$ and a string of digits $S$ in base $b$, the set
\[
A_{\lambda_f}(b,S):=\{\textup{$p$ prime : the first digits of $\lambda_f(p)$ in base $b$ are given by $S$}\}
\]
has logarithmic density equal to $\log_b(1+S^{-1})$.  Thus $\{\lambda_f(p)\}_{p\in\mathbb{P}}$ follows Benford's Law with respect to logarithmic density.  Both results rely on the now-proven Sato-Tate Conjecture.
\end{abstract}
\maketitle
\noindent

\section{Introduction and Statement of Results}

In 1881, astronomer Simon Newcomb \cite{SN} made the observation that certain pages of logarithm tables were much more worn than others.  Users of the tables referenced logarithms whose leading digit is 1 more frequently than other logarithms, contrary to the naive expectation that all of the logarithms would be referenced uniformly.  In 1938, Benford \cite{Benford} made a similar observation for a variety of sequences.

This bias, now known as Benford's Law, is given as follows.  Let $\mathbb{N}$ denote the set of positive integers, and let $\mathcal{I}\subset\mathbb{N}$ be an infinite subset.  Fix an integer base $b\geq2$ and a string of digits $S$ in base $b$.  For a given function $g:\mathbb{N}\to\mathbb{R}$, define
\begin{equation}
A_g(b,S):=\{\textup{$i\in \mathbb{N}$ : the first digits of $g(i)$ in base $b$ are given by $S$}\}.
\end{equation}
We define the {\it arithmetic density} of $A_g(b,S)$ within $\mathcal{I}$ by
\begin{equation}
\delta_{\mathcal{I}}(A_g(b,S))=\lim_{x\to\infty}\frac{\#\{i\leq x:i\in\mathcal{I}\cap A_g(b,S)\}}{\#\{i\leq x:i\in\mathcal{I}\}}.
\end{equation}
We say that the sequence $\{g(i)\}_{i\in \mathcal{I}}$ satisfies Benford's Law, or that $\{g(i)\}_{i\in \mathcal{I}}$ is Benford, if for any integer base $b\geq2$ and any string of digits $S$ in base $b$,
\begin{equation}
\delta_{\mathcal{I}}(A_g(b,S))=\log_b(1+S^{-1}).
\end{equation}
It is easy to show \cite{Diaconis} that $\{g(i)\}_{i\in \mathcal{I}}$ is Benford if and only if the set $\{\log_b(g(i)):i\in \mathcal{I}\}$ is equidistributed modulo 1 for each base $b$ (setting $\log_b(g(i))=0$ if $g(i)=0$).  For some general surveys on Benford's Law, we refer the reader to \cite{Hill1, Hill2, MTB, Raimi}.

Stirling's approximation of $\Gamma(s)$ in conjunction with standard equidistribution results quickly yields that $\{n!\}_{n\in\mathbb{N}}$ is a Benford sequence.  However, if $g(n)=n^a$ for any fixed $a\in\mathbb{R}$, then $\{g(n)\}_{n\in\mathbb{N}}$ is not a Benford sequence, for
\[
\limsup_{n\in\mathbb{N}}|n\log(|g(n+1)/g(n)|)|<\infty,
\]
contradicting equidistribution of $\log_b(g(n))$ modulo 1.  Taking $a=1$, one sees that the positive integers are not Benford.

Much real world data, including lengths of rivers, populations of nations, and heights of skyscrapers exhibit behavior which is suggestive of Benford's Law (when restricted to base 10).  There are also several settings in which Benford's law arises that are of arithmetic interest.  In \cite{BBH}, several dynamical systems such as linearly-dominated systems and non-autonomous dynamical systems are shown to be Benford.  Benford's Law is proven for distributions of values of $L$-functions \cite{KM} and the $3x+1$ problem \cite{KM,LS}.  In \cite{ARS}, the image of the partition function is shown to be Benford, as well as the coefficients of an infinite class of modular forms with poles.

While the positive integers are not Benford, we can say even more.  Specifically, if we take $g(n)=n$, the limits defining $\delta_{\mathbb{N}}(A_g(b,1))$ and $\delta_{\mathbb{N}}(A_g(2,10))$ for any $b\geq3$ do not exist.  For example, one easily sees that
{\small
\[
\liminf_{x\to\infty}\frac{\#\{i\leq x:i\in\mathbb{N}\cap A_g(10,1)\}}{\#\{i\leq x:i\in\mathbb{N}\}}=\frac{1}{9},\quad\limsup_{x\to\infty}\frac{\#\{i\leq x:i\in\mathbb{N}\cap A_g(10,1)\}}{\#\{i\leq x:i\in\mathbb{N}\}}=\frac{5}{9}.
\]}%
However, if we change our notion of density, the first digits of the integers still satisfy the distribution in Benford's law.  We define the {\it logarithmic density} of $A_g(b,S)$ in $\mathcal{I}$ by
{\small
\begin{equation}
\widetilde{\delta}_{\mathcal{I}}(A_g(b,S))=\lim_{x\to\infty}\frac{\displaystyle\sum_{\substack{i\leq x \\ i\in\mathcal{I}\cap A_g(b,S)}}i^{-1}}{\displaystyle\sum_{\substack{i\leq x \\ i\in\mathcal{I}}}i^{-1}}.
\end{equation}}%
With this modified notion of density, we have that $\widetilde{\delta}_{\mathbb{N}}(A_g(b,S))$ exists and equals $\log_b(1+S^{-1})$ for any base $b\geq2$ and any string $S$ in base $b$ \cite{Duncan}.  In light of this fact, we say that a sequence $\{g(i)\}_{i\in\mathcal{I}}$ is {\it logarithmically Benford} if
\begin{equation}
\widetilde{\delta}_{\mathcal{I}}(A_g(b,S))=\log_b(1+S^{-1})
\end{equation}
for any base $b$ and any string $S$ in base $b$.  We note that if a set has an arithmetic density, then it also has a logarithmic density, and the two densities are equal.  Thus all Benford sequences are logarithmically Benford.

\begin{remark}
\textup{Logarithmic density is closely related to Dirichlet density, which is ubiquitous in number theory.  For integer sequences, logarithmic density and Dirichlet density have equivalent definitions; this (and much more) is proven in Part 3 of \cite{Tenenbaum}.  For a discussion on Dirichlet density in the context of the prime number theorem for arithmetic progressions or the Chebotarev density theorem, see Chapter 7 of \cite{Neukirch}.}
\end{remark}
Since the $n$-th prime is asymptotically equal to $n\log(n)$ by the prime number theorem, one sees that the primes are not Benford.  However, Whitney \cite{Whitney} proved that for $g(n)=n$, we have $\widetilde{\delta}_{\mathbb{P}}(A_g(10,S))=\log_{10}(1+S^{-1})$ for any string $S$ in base 10.  The fact that the primes are logarithmically Benford follows easily from Whitney's proof.  Serre briefly discusses this problem for the primes in Chapter 4, Section 5 of \cite{arithmetic}.

In this paper, we consider sequences given by Fourier coefficients of certain modular forms without complex multiplication (see \cite{Ono}).  Specifically, let 
\begin{equation}
f(z)=\sum_{n=1}^\infty \lambda_f(n)q^n\in S_k^{\textup{new}}(\Gamma_0(N)),\quad q=e^{2\pi iz}
\end{equation}
be a newform (i.e., a holomorphic cuspidal normalized Hecke eigenform) of even weight $k$ and trivial nebentypus on $\Gamma_0(N)$ that does not have complex multiplication.  The Fourier coefficients $\lambda_f(n)$ of such a newform will be real.  We consider sets of the form
\[
A_{\lambda_f}(b,S)=\{\textup{$n\in\mathbb{N}$ : the first digits of $\lambda_f(n)$ in base $b$ are given by $S$}\}.
\]

One important example of the newforms under our consideration is the weight 12 newform on $\Gamma_0(1)$ given by
\[
\Delta(z)=q\prod_{n=1}^\infty(1-q^n)^{24}=\sum_{n=1}^\infty\tau(n)q^n,
\]
where $\tau(n)$ is the Ramanujan tau function.  Consider the following table.

\begin{table}[H]
\begin{tabular}{|c|c|c|c|}\hline
$x$ & $\#\{p\leq x:\textup{the first digit of $\tau(p)$ is $1$}\}/\pi(x)$  \\ \hline
$10^3$ & $0.28571\ldots$  \\ \hline
$10^4$ & $0.29454\ldots$  \\ \hline
$10^5$ & $0.29993\ldots$  \\ \hline
\end{tabular}
\end{table}

If $\{\tau(p)\}_{p\in\mathbb{P}}$ were Benford, then we would have $\delta_{\mathbb{P}}(A_{\tau}(10,1))=\log_{10}(2)\approx 0.30103$.  If we only consider the above table, then it seems to be the case that $\delta_{\mathbb{P}}(A_{\tau}(10,1))$ exists and might very well equal $\log_{10}(2)$.  However, the plot in Figure 1 indicates that this conclusion is very far from the truth.

\begin{figure}[H]
\centering
\includegraphics[width=0.6\textwidth]{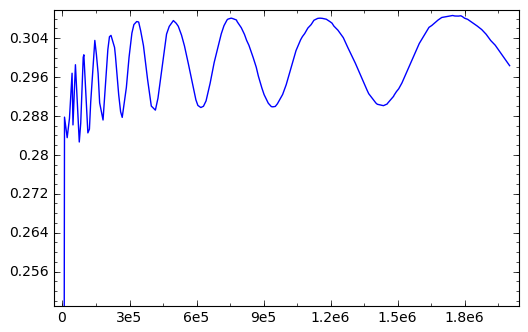}
\caption{The proportion of primes $p\leq x$ for which $\tau(p)$ has leading digit 1 for $x\leq 2\times10^6$.}
\end{figure}

It turns out to be the case that the arithmetic density $\delta_{\mathbb{P}}(A_{\tau}(10,1))$ does not exist.  However, $\widetilde{\delta}_{\mathbb{P}}(A_{\tau}(10,1))$ does exist, and it equals $\log_{10}(2)$.  More generally, we prove the two following results.

\begin{theorem}
\label{no-benford}
Let $f(z) = \sum_{n=1}^\infty \lambda_f(n)q^n\in S_k^{\textup{new}}(\Gamma_0(N))$ be a newform of even weight $k\geq2$ without complex multiplication.  The arithmetic density $\delta_{\mathbb{P}}(A_{\lambda_f}(b,1))$ does not exist for any integer base $b\geq3$, and the arithmetic density $\delta_{\mathbb{P}}(A_{\lambda_f}(2,10))$ does not exist.  Thus the sequence $\{\lambda_f(p)\}_{p\in\mathbb{P}}$ is not Benford.
\end{theorem}

\begin{remark}
For a string $S$ in base $b$, the method used in the proof of Theorem \ref{no-benford} can be modified to show that there are infinitely many primes $p$ such that $\lambda_f(p)$ begins with $S$. Since this fact is also a direct consequence of Theorem \ref{log-density}, we omit the details.
\end{remark}

\begin{theorem}
\label{log-density}
Let $f(z) = \sum_{n=1}^\infty \lambda_f(n)q^n\in S_k^{\textup{new}}(\Gamma_0(N))$ be a newform of even weight $k\geq2$ without complex multiplication.  Let $b\geq2$ be a given integer base, and let $S$ be an initial string of digits in base $b$.  We have $\widetilde{\delta}_{\mathbb{P}}(A_{\lambda_f}(b,S))=\log_b(1+S^{-1})$.  Thus $\{\lambda(p)\}_{p\in\mathbb{P}}$ is logarithmically Benford.
\end{theorem}

\subsection*{Acknowledgements}

The authors thank Ken Ono and the anonymous referee for their comments and Ken Ono for suggesting this project.  The authors used Maple 18, Mathematica 9, and SAGE for the numerical computations and plots.

\section{The Sato-Tate Conjecture}

Let 
\[
f(z)=\sum_{n=1}^\infty \lambda_f(n)q^n\in S_k^{\textup{new}}(\Gamma_0(N))
\]
be a newform of even weight $k\geq2$ and trivial nebentypus on $\Gamma_0(N)$.  The Fourier coefficients $\lambda_f(n)$ will lie in the ring of integers of a totally real number field.  By Deligne's proof of the Weil conjectures, we have that for every prime $p$,
\[
|\lambda_f(p)|\leq 2p^{\frac{k-1}{2}}.
\]
Thus there exists $\theta_p\in[0,\pi]$ satisfying
\[
\lambda_f(p)=2p^{\frac{k-1}{2}}\cos\theta_p.
\]

Around 1960, Sato and Tate studied the sequence $\{\cos\theta_p\}$ as $p$ varies through the primes when $f$ is the newform associated to an elliptic curve $E/\mathbb{Q}$ without complex multiplication.  All such newforms have weight $k=2$, and if $p$ is prime, then
\[
\lambda_f(p)=p+1-\#E(\mathbb{F}_p),\qquad|\lambda_f(p)|\leq2\sqrt{p}.
\]
where $\#E(\mathbb{F}_p)$ is the number of $\mathbb{F}_p$-rational points on $E/\mathbb{Q}$.  Thus $\cos\theta_p$ is the normalized error in approximating $\#E(\mathbb{F}_p)$ with $p+1$.  Sato and Tate conjectured a distribution for the sequence $\{\cos\theta_p\}$, and this conjecture was later generalized to a much larger class of newforms.  This conjecture, which we now state, was proven by Barnet-Lamb, Geraghty, Harris, and Taylor \cite{sato-tate}.

\begin{theorem}[The Sato-Tate Conjecture]
Let $f(z)=\sum_{n=1}^\infty \lambda_f(n)q^n\in S_k^{\textup{new}}(\Gamma_0(N))$ be a newform of even weight $k\geq2$ without complex multiplication.  The sequence $\{\cos\theta_p\}$ is equidistributed in the interval $[-1,1]$ with respect to the measure
\[
d\mu_{ST}=\frac{2}{\pi}\sqrt{1-t^2}~dt.
\]
In other words, if $I\subset[-1,1]$ is a subinterval and we define
\[
\pi_{f,I}(x)=\#\{p\leq x:\cos\theta_p\in I\},
\]
then as $x\to\infty$,
\[
\pi_{f,I}(x)\sim\mu_{ST}(I)\pi(x).
\]
\end{theorem}

The following immediate corollary of the Sato-Tate Conjecture plays an important role in the proof of Theorem \ref{log-density}.

\begin{corollary}
\label{cor}
Let $f(z)=\sum_{n=1}^\infty \lambda_f(n)q^n\in S_k^{\textup{new}}(\Gamma_0(N))$ be a newform of even weight $k\geq2$ without complex multiplication.  Let $I\subset[-1,1]$ be an interval.  As $x\to\infty$, we have
\[
\sum_{\substack{p\leq x \\ \cos\theta_p\in I}}p^{-1}\sim\mu_{ST}(I)\sum_{p\leq x}p^{-1}.
\]
\end{corollary}

\section{Proof of theorem \ref{no-benford}}

To prove Theorem~\ref{no-benford}, we use the Sato-Tate Conjecture to construct many large intervals on which the proportion of primes $p$ for which $\lam_f(p)$ has leading digit $1$ in a given base $b\geq3$ differs from the Benford expectation.  (For base $b=2$, we use the leading digits 10 because all nonzero real numbers have leading digit 1 in their base 2 expansion.) This shows that $\{\lam_f(p)\}_{p\in\mathbb{P}}$ is not Benford in any base. To do this, we first state a lemma about the Sato-Tate measures of certain intervals.

\begin{lemma} \label{difference-bound}
Fix $b\geq 3$ and let $c$ be a sufficiently large positive integer. For $d=1,2$, set
\[
I_{d,1}(c) = \bigcup_{j\in\mathbb{Z}} \left[\frac{b^{-j}}{db^c-2},\frac{2b^{-j}}{db^c-1}\right] \cap [0,1].
\]
Then
\[
2\mu_{ST}(I_{2,1}(c)) - 2\mu_{ST}(I_{1,1}(c)) > \frac{1}{40}.
\]
\end{lemma}

\begin{proof}[Proof of Lemma \ref{difference-bound}]
If $c>2$, then
{\small
\[
\frac{1}{b} < \frac{b^{c-1}}{b^c-2} < \frac{2b^{c-1}}{b^c-1} < 1 \qquad \text{and} \qquad \frac{1}{b} < \frac{b^{c}}{2b^c-2} < 1 < \frac{2b^{c}}{2b^c-1}.
\]}%
Thus
{\small
\[
I_{1,1}(c) =  \bigcup_{j\geq -c+1} \left[\frac{b^{-j}}{b^c-2},\frac{2b^{-j}}{b^c-1}\right] = \bigcup_{m\geq 0} \left[\frac{b^{-1-m}}{1-2b^{-c}},\frac{2b^{-1-m}}{1-b^{-c}}\right]
\]}%
and
{\small
\begin{align*}
I_{2,1} (c)&= \left[\frac{b^{c}}{2b^c-2},1\right] \cup \bigcup_{j>-c} \left[\frac{b^{-j}}{2b^c-2},\frac{2b^{-j}}{2b^c-1}\right]\\
&= \left[\frac{1}{2-2b^{-c}},1\right] \cup \bigcup_{m>0} \left[\frac{b^{-m}}{2-2b^{-c}},\frac{2b^{-m}}{2-b^{-c}}\right].
\end{align*}}%
For all $b\geq3$, we have
{\small
\begin{align*}
&\lim_{c\to\infty}\left(\mu_{ST}\left(\left[\frac{1}{2-2b^{-c}},1\right]\right) - \mu_{ST}\left(\left[\frac{b^{-1}}{1-2b^{-c}},\frac{2b^{-1}}{1-b^{-c}}\right]\right)\right)\\
=\;& 2\mu_{ST}\left(\left[1/2,1\right]\right) - 2\mu_{ST}\left(\left[b^{-1},2b^{-1}\right]\right)>\frac{1}{40}.
\end{align*}}%
Thus for all sufficiently large $c$, we have
{\small
\[
2\mu_{ST}\left(\left[\frac{1}{2-2b^{-c}},1\right]\right) - 2\mu_{ST}\left(\left[\frac{b^{-1}}{1-2b^{-c}},\frac{2b^{-1}}{1-b^{-c}}\right]\right)>\frac{1}{40}
\]}
Similarly, we have
{\small
\[
2\mu_{ST}\left(\left[\frac{b^{-m}}{2-2b^{-c}},\frac{2b^{-m}}{2-b^{-c}}\right]\right) - 2\mu_{ST}\left( \left[\frac{b^{-1-m}}{1-2b^{-c}},\frac{2b^{-1-m}}{1-b^{-c}}\right]\right)>0
\]}%
for all $m>0,$ and the result follows.
\end{proof}

Now we may prove Theorem \ref{no-benford}.

\begin{proof}[Proof of Theorem \ref{no-benford}]

Let $f$ be a newform as in the statement of the theorem and let $b\geq 3$.  In order prove that the arithmetic density $\delta_{\mathbb{P}}(A_{\lambda_f}(b,S))$ does not exist, it suffices to show that for some fixed $\beta>1$, the value of
\[
\lim_{n\to\infty}\frac{\#\{\alpha \beta^n \leq p < \gamma \beta^n : p \in A_{\lambda_f}(b,1)\}}{\#\{\alpha \beta^n \leq p < \gamma \beta^n\}}
\]
varies with different choices of values of $\alpha$ and $\gamma$ with $\alpha<\gamma$.

We start with some preliminary constructions.  For $S\in\{1,\ldots,b-1\}$, define
\begin{align*}
I_{1,S}(c) &= \bigcup_{j\in\mathbb{Z}} \left[\frac{Sb^{-j}}{b^c-2},\frac{(S+1)b^{-j}}{b^c-1}\right] \cap [0,1],\\
I_{2,S}(c) &= \bigcup_{j\in\mathbb{Z}} \left[\frac{Sb^{-j}}{2b^c-2},\frac{(S+1)b^{-j}}{2b^c-1}\right] \cap [0,1].
\end{align*}
Fix $0<\epsilon<1/40$, and let $c$ be a sufficiently large positive integer so that Lemma \ref{difference-bound} holds and
\[
\mu_{ST}\left([0,1] - \cup_{S=1}^{b-1}I_{d,S}(c)\right) < \epsilon/4
\]
for $d=1,2$.  Let $\beta=b^{\frac{2}{k-1}}$, $\alpha_1 =(\frac{b^c-2}{2})^{\frac{2}{k-1}}$, and $\gamma_1 = (\frac{b^c-1}{2})^{\frac{2}{k-1}}$.  We consider the primes $p$ such that
\begin{equation}\label{boundforp}
(b^c-2)b^n\leq 2p^{\frac{k-1}{2}}< (b^c-1)b^n.
\end{equation}
Note that if $p$ is bounded as in \eqref{boundforp} and $|\!\cos\theta_p| \in I_{1,S},$ then
\[
|\lambda_f(p)| \in [Sb^{n-j},(S+1)b^{n-j})
\]
for some $j \in \mathbb{Z}$; that is, its first digits are given by $S$.  By letting $c$ be sufficiently large and setting $S=1$, the Sato-Tate Conjecture implies that
\begin{equation}\label{proportion1}
\left|\lim_{n\to\infty}\frac{\#\{\alpha_1 \beta^n \leq p < \gamma_1 \beta^n : p \in A_{\lambda_f}(b,1)\}}{\#\{\alpha_1 \beta^n \leq p < \gamma_1 \beta^n\}} -2\mu_{ST}(I_{1,1}(c))\right|< \frac{\epsilon}{2}.
\end{equation}

Similarly, by letting $\alpha_2 =(\frac{2b^c-2}{2})^{\frac{2}{k-1}}$ and $\gamma_2 = (\frac{2b^c-1}{2})^{\frac{2}{k-1}}$, we find that
\begin{equation}\label{proportion2}
\left|\lim_{n\to\infty}\frac{\#\{\alpha_2 \beta^n \leq p < \gamma_2 \beta^n : p \in A_{\lambda_f}(b,1)\}}{\#\{\alpha_2 \beta^n \leq p < \gamma_2 \beta^n\}} - 2\mu_{ST}(I_{2,1}(c))\right|< \frac{\epsilon}{2}.
\end{equation}

Now, suppose to the contrary that $\delta_{\mathbb{P}}(A_{\lambda_f}(b,1))$ exists.  It follows from \eqref{proportion1} and \eqref{proportion2} that
\[
\left| 2\mu_{ST}(I_{1,1})(c) - 2\mu_{ST}(I_{2,1})(c)\right| < \epsilon < \frac{1}{40},
\]
which contradicts Lemma \ref{difference-bound}.  The theorem now follows for bases $b\geq3$.  For $b=2$, one arrives at the same conclusion as for $b\geq3$ by comparing $I_{d,10}(c)$ for $d=1,3$ in Lemma \ref{difference-bound}.  The computations for $b=2$ are essentially the same.
\end{proof}

\section{Proof of Theorem~\ref{log-density}}



For this section, $p$ will always denote a prime.  Let $f$ be a newform satisfying the hypotheses of Theorem \ref{log-density}.  Let $b\geq2$ be a base, and let $S$ be a string of digits in base $b$.  By the definition of a logarithmically Benford sequence and the estimate
\begin{equation}
\label{sum-recip}
\sum_{p\le x}p^{-1}\sim\log\log x,
\end{equation}
a proof of Theorem \ref{log-density} will follow from proving that as $x\to\infty$,
\begin{equation}
\label{goal}
\sum_{\substack{p\leq x \\ p\in A_{\lambda_f}(b,S)}}p^{-1}\sim\log_b(1+S^{-1})\log\log x.
\end{equation}
This will be a consequence of the following key lemma.
\begin{lemma}
\label{big-inequality}
Let $f(z) = \sum_{n=1}^\infty \lambda_f(n)q^n\in S_k^{\textup{new}}(\Gamma_0(N))$ be a newform of even weight $k\geq2$ without complex multiplication.  Let $b\geq2$ be a given base, let $S$ be an initial string of digits in base $b$, and let $\ell> \max\{S,40\}$ be an integer.  As $x\to\infty$, we have
\begin{align*}
&(1+o(1))(\log_b(1+S^{-1})-\log(1+\ell^{-1}))\log\log x\\
\leq\;&\sum_{\substack{p\le x \\ p\in A_{\lambda_f}(b,S) \\ |\!\cos\theta_p|>\ell^{-1}}}p^{-1}\\
\leq\;&(1+o(1))(\log_b(1+S^{-1})+\log(1+\ell^{-1}))\log\log x+2\log\log\ell.
\end{align*}
\end{lemma}
\begin{proof}
We prove the upper bound; the lower bound is proven similarly.  Writing $\lambda_f(p)=2p^{\frac{k-1}{2}}\cos\theta_p$, we first observe that
{\small
\begin{align*}
\sum_{\substack{p\le x \\ p\in A_{\lambda_f}(b,S) \\ |\!\cos\theta_p|>\ell^{-1}}}p^{-1}&=\sum_{t=-\infty}^{\infty}\sum_{i=\ell}^{\ell^2-1}\sum_{\substack{p\le x\\ S\cdot b^t\le |\lambda_f(p)|<(S+1)\cdot b^t\\ \frac{i}{\ell^2}<|\!\cos\theta_p|\le \frac{i+1}{\ell^2}}}p^{-1}\\
&\leq \sum_{t=-\infty}^{\infty}\sum_{i=\ell}^{\ell^2-1}\sum_{\substack{p\leq x \\ \left(\frac{S\ell^{2}}{i+1} b^t\right)^{\frac{2}{k-1}}\le p\le\left(\frac{(S+1)\ell^{2}}{i} b^t\right)^{\frac{2}{k-1}}\\ \frac{i}{\ell^{2}}<|\!\cos\theta_p|\le \frac{i+1}{\ell^{2}}}}p^{-1}.
\end{align*}}%
To bound the contribution when $t<0$, all of the primes in the sum are at most $(\frac{(S+1)\ell}{b})^{\frac{2}{k-1}}$; since $\ell>\max\{S,40\}$, we have
\[
\sum_{t=-\infty}^{-1}\sum_{i=\ell}^{\ell^2-1}\sum_{\substack{p\leq x \\ \left(\frac{S\ell^{2}}{i+1} b^t\right)^{\frac{2}{k-1}}\le p\le\left(\frac{(S+1)\ell^{2}}{i} b^t\right)^{\frac{2}{k-1}}\\ \frac{i}{\ell^{2}}<|\!\cos\theta_p|\le \frac{i+1}{\ell^{2}}}}p^{-1}\leq\sum_{p\leq \left(\frac{(S+1)\ell}{b}\right)^{\frac{2}{k-1}}}p^{-1}\leq 2\log\log \ell.
\]

To bound the contribution when $t\geq0$, fix $\ell\leq i\leq\ell^2-1$.  Recall that we may write $\lambda_f(p)=2p^{\frac{k-1}{2}}\cos\theta_p$.   If $(\frac{(S+1)\ell^2}{i} b^t)^{\frac{2}{k-1}}\le x$, then Corollary \ref{cor} implies that
\begin{align*}
&\sum_{\substack{p\le x\\ S\cdot b^t\le |\lambda_f(p)|<(S+1)\cdot b^t\\ \frac{i}{\ell^2}<|\!\cos\theta_p|\le \frac{i+1}{\ell^2}}}p^{-1}\\
\le\;&\sum_{\substack{\left(\frac{S\ell^{2}}{i+1} b^t\right)^{\frac{2}{k-1}}\le p\le\left(\frac{(S+1)\ell^{2}}{i} b^t\right)^{\frac{2}{k-1}}\\ \frac{i}{\ell^{2}}<|\!\cos\theta_p|\le \frac{i+1}{\ell^{2}}}}p^{-1}\\
=\;&2(1+o(1))\mu_{ST}\left(\left[\frac{i}{\ell^{2}},\frac{i+1}{\ell^{2}}\right]\right)\sum_{\left(\frac{S\ell^{2}}{i+1} b^t\right)^{\frac{2}{k-1}}\le p\le\left(\frac{(S+1)\ell^{2}}{i} b^t\right)^{\frac{2}{k-1}}}p^{-1}\\
=\;&2(1+o(1))\mu_{ST}\left(\left[\frac{i}{\ell^{2}},\frac{i+1}{\ell^{2}}\right]\right)\log\left(\frac{\log_b(\frac{\ell^{2}}{i})+\log_b(S+1)+t}{\log_b(\frac{\ell^{2}}{i+1})+\log_b(S)+t}\right).
\end{align*}

Setting $B_0=\log_b(\frac{\ell^{2}}{i+1})+\log_b(S)$ and $B_1=\log_b(\frac{\ell^{2}}{i})+\log_b(S+1)$, we have for any large $N$ that
{\small
\[
\sum_{t=0}^N\sum_{\left(\frac{S\ell^{2}}{i+1} b^t\right)^{\frac{2}{k-1}}\le p\le\left(\frac{(S+1)\ell^{2}}{i} b^t\right)^{\frac{2}{k-1}}}p^{-1}\sim\log\prod_{t=0}^N\frac{B_1+t}{B_0+t}.
\]}%

Switching the order of summation, we obtain the inequality

{\small
\begin{align*}
&\sum_{t=0}^{\infty}\sum_{i=\ell}^{\ell^2-1}\sum_{\substack{p\le x\\ S\cdot b^t\le |\lambda_f(p)|<(S+1)\cdot b^t\\ \frac{i}{\ell^2}<|\!\cos\theta_p|\le \frac{i+1}{\ell^2}}}p^{-1}\\
\leq\;&\sum_{i=\ell}^{\ell^{2}-1}2\mu\left(\left[\frac{i}{\ell^{2}},\frac{i+1}{\ell^{2}}\right]\right)(1+o(1))\sum_{t=0}^{\infty}\sum_{\left(\frac{S\ell^{2}}{i+1} b^t\right)^{\frac{2}{k-1}}\le p\le\min\left\{\left(\frac{(S+1)\ell^{2}}{i} b^t\right)^{\frac{2}{k-1}},x\right\}}p^{-1}\\
\leq\;&\sum_{i=\ell}^{\ell^{2}-1}2\mu\left(\left[\frac{i}{\ell^{2}},\frac{i+1}{\ell^{2}}\right]\right)(1+o(1))\log\left(\prod_{0\le t\le\log_b\left(\frac{ix^{\frac{k-1}{2}}}{(S+1)\ell^{2}}\right)}\frac{B_1+t}{B_0+t}\right)+2\log\log \ell.
\end{align*}}%
Using Euler's formula for the Gamma function
\[
\Gam(z)=\lim_{n\to\infty}\frac{n!n^z}{\prod_{i=0}^n(z+i)},
\]
we find that the contribution from $t\geq0$ is at most
\begin{align*}
&\sum_{i=\ell}^{\ell^{2}-1}2\mu\left(\left[\frac{i}{\ell^{2}},\frac{i+1}{\ell^{2}}\right]\right)(1+o(1))\left(\log\frac{\Gam(B_0)}{\Gam(B_1)}+(B_1-B_0)\log\log x\right)\\
\leq\;&(1+o(1))\left(\log_b\left(1+\ell^{-1}\right)+\log_b\left(1+S^{-1}\right)\right)\log\log x.
\end{align*}
This proves the claimed upper bound.  Using the inequality
\[
\sum_{\substack{p\le x\\ S\cdot b^t\le 2p^{\frac{k-1}{2}}|\!\cos\theta_p| <(S+1)\cdot b^t\\ \frac{i}{\ell}<|\!\cos\theta_p|\le \frac{i+1}{\ell}}}p^{-1}
\ge\sum_{\substack{\left(\frac{S\ell}{i} b^t\right)^{\frac{2}{k-1}}\le p\le\left(\frac{(S+1)\ell}{i+1} b^t\right)^{\frac{2}{k-1}}\\ \frac{i}{\ell}<|\!\cos\theta_p|\le \frac{i+1}{\ell}}}p^{-1}
\]
for $\ell\leq i\leq \ell^2-1$, the lower bound is proven similarly.
\end{proof}

\begin{proof}[Proof of Theorem \ref{log-density}]
Let $0<\epsilon<\log_b(2)$, let $\ell>\max\{\frac{1}{b^\epsilon-1},S,40\}$ be an integer, and let $x>\exp((\log\ell)^{2/\epsilon})$.  If we write $\lambda_f(p)=2p^{\frac{k-1}{2}}\cos\theta_p$ with $\theta_p\in[0,\pi]$, then
\begin{equation}
\label{initial}
\sum_{\substack{p\le x\\ p\in A_f(b,S)}}p^{-1}=\sum_{\substack{p\le x \\ p\in A_{\lambda_f}(b,S) \\ |\!\cos\theta_p|\leq\ell^{-1}}}p^{-1}
+\sum_{\substack{p\le x \\ p\in A_{\lambda_f}(b,S) \\ |\!\cos\theta_p|>\ell^{-1}}}p^{-1}.
\end{equation}
By Corollary \ref{cor}, the first term is at most
\begin{equation}
\label{term1}
\sum_{\substack{p\le x\\ 0\le|\!\cos\theta_p|\le\ell^{-1}}}p^{-1}=(2+o(1))\mu_{ST}\left(\left[0,\ell^{-1}\right]\right)\log\log x.
\end{equation}
Using Lemma \ref{big-inequality}, we now have
{\small
\begin{align*}
&(1+o(1))(\log_b(1+S^{-1})-\log_b(1+\ell^{-1}))\log\log x\\
\leq\;&\sum_{\substack{p\le x\\ p\in A_f(b,S)}}p^{-1}\\
\leq\;&(1+o(1))(\log_b(1+S^{-1})+\log_b(1+\ell^{-1})+2\mu_{ST}([0,\ell^{-1}]))\log\log x+2\log\log \ell.
\end{align*}}%
Thus
\begin{align*}
&(1+o(1))(\log_b(1+S^{-1})-\epsilon)\log\log x \\
\leq\;&\sum_{\substack{p\le x\\ p\in A_f(b,S)}}p^{-1}\\
\leq\;&(1+o(1))(\log_b(1+S^{-1})+9\log(b)\epsilon)\log\log x.
\end{align*}
Letting $\epsilon\to0$, we obtain (\ref{goal}), as desired.
\end{proof}

\bibliography{biblio}
\end{document}